\documentclass[10pt,a4paper]{article}
\pdfoutput=1
\usepackage[utf8]{inputenc}
\usepackage{amsmath}
\usepackage{amsfonts}
\usepackage{amssymb}
\usepackage{amsmath}
\usepackage{amsthm}
\usepackage{amssymb}
\usepackage{pdfpages}
\usepackage[shortlabels]{enumitem}	
\usepackage{xcolor}
\newtheorem{thm}{Theorem}[section]
\newtheorem{theorem}{Theorem}
\newtheorem{cor}[thm]{Corollary}
\newtheorem{lm}[thm]{Lemma}
\newtheorem{pr}[thm]{Proposition}
\theoremstyle{definition}
\newtheorem{df}[thm]{Definition}
\theoremstyle{remark}
\newtheorem{rem}[thm]{Remark}
\usepackage[a4paper, left = 2.5cm, right = 2.5cm, top = 2.5cm, bottom = 2.5cm, headsep = 1.2cm]{geometry}
\author{Karol Duda \\ \\ with an appendix by Karol Duda and Aleksander Ivanov}
\title{Amenability and computability}
\date{}
\begin{document}
\maketitle

\begin{abstract}

In this paper we extend the approach of M. Cavaleri to effective amenability 
to the class of computably enumerable groups, 
i.e. in particular we do not assume that groups are finitely generated. 
The main results of the paper concern some new directions in this approach. 
In the case of computable groups we study decidability of amenability of finitely generated subgroups, 
complexity of the set of all effective F\o lner sequences and effective paradoxical decomposition.

In the appendix we attach a version of the paper "On decidability of amenability in computable groups" by K.  Duda and A. Ivanov which has been already published. 
\end{abstract}

\section{Introduction}

M. Cavaleri has shown in \cite{MC3} that every amenable finitely generated recursively presented group has computable Reiter functions and subrecursive F\o lner functions.
 Moreover, for a finitely generated recursively presented group with solvable Word Problem, amenability is equivalent to these conditions and in fact it is equivalent 
to so called effective amenability. 
The latter means existence of an algorithm which finds $n$-F\o lner sets for all $n$.

Since being finitely generated is not necessary for amenability, the question arises what happens if we consider the case of recursively presented groups without the assumption of finite generation.
According the approach of computable algebra the question concerns the class of computably enumerable groups and the subclass of computable groups, which corresponds to decidability of the Word Problem.

The following Theorem generalizes some results of Cavaleri to a case of computably enumerable groups:

\begin{theorem} \label{1} 
Let $G$ be a computably enumerable group.
The following conditions are equivalent:
\begin{enumerate}[(i)]
\item $G$ is amenable;
\item $G$ has computable Reiter functions;
\item $G$ has subrecursive F\o lner function.
\item $G$ is $\Sigma$-amenable (see Definition \ref{cea}).
\end{enumerate}
Moreover, computable amenability of $G$ implies computability of it.
\end{theorem} 

A paradoxical decomposition of a group is a triple $(K,(A_k)_{k\in K}, (B_k)_{k\in K})$ 
consisting families $A$ and $B$ of subsets of $G$ indexed by elements of a finite set $K\subset G$ such that:
$$
G = \Big(\bigsqcup\limits_{k\in K}kA_k\Big)\bigsqcup\Big(\bigsqcup\limits_{k\in K}kB_k\Big)=\Big(\bigsqcup\limits_{k\in K}A_k\Big)=\Big(\bigsqcup\limits_{k\in K}B_k\Big).
$$
Here we use the definition given in \cite{csc} where some members $A_k$ or $B_k$ can be empty. 
It is equivalent to the traditional one. 
Thus the existence of such a paradoxical decomposition is  opposite to amenability. 
By demanding families $(A_k )$ and $(B_k )$ to consist of computable sets, we introduce an effective paradoxical decomposition. 
Using an effective version of the Hall's Harem Theorem we prove the following theorem. 

\begin{theorem} \label{2} 
Let $G$ be a computable group. 
There is an effective procedure which given $K_0 \subset G$ 
such that for some natural $n$ there is no $n$-F\o lner set with respect to $K_0$, 
finds a finite $K$ with an effective paradoxical decomposition of $G$ as above.
\end{theorem}

\noindent We call such a set $K_0$ a {\bf witness} of the Banach-Tarski paradox. 
The question arises, how complex is the family (denoted by $\mathfrak{W}_{BT}$) of such subsets of a computable group? We prove the following theorem. 

\begin{theorem} \label{3}
For any computable group the family $\mathfrak{W}_{BT}$ 
belongs to the class $\Sigma^{0}_{2}$. 
In the case of the fully residually free groups the family $\mathfrak{W}_{BT}$ is computable.
\end{theorem} 
\noindent After this theorem a principal question arises if there is a computable group for which 
the family $\mathfrak{W}_{BT}$ is not computable. 
Moreover it is worth mentioning that the latter condition is equivalent to undecidability of the problem if a finite subset 
generates an amenable subgroup.  
The appendix of this paper gives a required example.  
Using it we also build a finitely presented group 
with decidable word problem where the family $\mathfrak{W}_{BT}$ is not computable. 
This shows that the statement of Theorem \ref{3} cannot be extended to 
finitely presented groups with decidable word problem.

The paper is organized as follows.
Section 2 contains some basic definitions and preliminary observations.
In Sections 3 - 4 we generalize Cavaleri's characterizations 
(using very similar arguments) of some versions of effective amenability to the case of computably enumerable groups. 
In these sections we prove Theorem \ref{1}.
In Section 5 we study complexity of the set of effective F\o lner sequences for computable groups. 
Sections 6 - 8 are dedicated to the effectiveness of a paradoxical decomposition. 
In Section 6 we introduce and prove an effective version of the Hall's Harem Theorem. 
We use it to prove Theorem \ref{2}  in Section 7. 
In Section 8 we introduce a notion of a witnesses of the Banach-Tarski paradox and study the complexity of the set of witnesses (Theorem \ref{3}). 

We mention papers of I. Bilanovic, J. Chubb and S. Roven \cite{bcr}
and of A. Darbinyan \cite{darar} as other papers in the field (applied to other group-theoretic properties).

The material of this paper is based on the master thesis of the author, written under supervision of Aleksander Ivanov. 
The author is grateful to him for support. 
The author is grateful to M. Cavaleri and T. Ceccherini-Silberstein for reading the paper and helpful remarks. In particular, the idea of Proposition \ref{fif} belongs to M. Cavaleri. 

\section{Preliminaries}
From now on we identify each finite set $F\subset \mathbb{N}$ with its  G\" odel number.
For any sets $X$ and $Y$ we will write $X \subset\subset Y$ to denote that $X$ is a finite subset of $Y$.  
For any $i\in\mathbb{N}$, we denote the set $\{1,2,\ldots, i\}$ by $[i]$. 
Throughout this paper, $G$ is a countable group without any presumption about its generating set.

\subsection{Computability}

We use standard material from the computability theory (see  \cite{sri}).
A function is \textbf{subrecursive} if it admits a computable total upper bound.
Sequence $(n_i)_{i\in\mathbb{N}}$ of natural numbers is called \textbf{effective}, if the function $k\rightarrow n_k$ is recursive.

Let $G$ be a countable group generated by some $X\subseteq G$. 
The group $G$ is called {\bf recursively presented} (see Section IV.3 in \cite{ls}) if 
$X$ can be identified with $\mathbb{N}$ (or with some $\{ 0,\ldots , n\}$) 
so that $G$ has a recursively enumerable set of relators in $X$.  
Below we give an equivalent definition, see Definition \ref{df1}. 
It is justified by a possibility identification of the whole $G$ with $\mathbb{N}$. 
We follow the approach of \cite{khmi}. 

\begin{df}\label{df0}
Let $G$ be a group and $\nu: \mathbb{N} \rightarrow G$ be a surjective function. 
We call the pair $(G,\nu)$ a \textbf{numbered group}.
The function $\nu$ is called a \textbf{numbering} of $G$.
If $g\in G$ and $\nu(n)=g$, then $n$ is called a number of $g$.
\end{df}

\begin{df}\label{comp_gr}
A numbered group $(G,\nu)$ has a \textbf{computable presentation} 
if $\nu$ is a bijection and 
the set 
$$
\mathsf{MultT} := \{(i,j,k): \quad \nu(i)\nu(j)=\nu(k)\}$$ 
is computable (=  decidable).
\end{df}

Any finitely generated group with decidable word problem 
obviously has a computable presentation. 
This also holds in the case of the free group $\mathbb{F}_{\omega}$ 
with the free basis $\{ x_0 , \ldots , x_i , \ldots \}$. 
If we fix a computable presentation $(\mathbb{F}_{\omega}, \nu_F )$
then for every recursively presented group $G= \langle X\rangle$ 
and a natural homomorphism $\rho : \mathbb{F}_{\omega} \rightarrow G$ 
(taking $\omega$ onto $X$) we obtain a numbering 
$\nu = \rho \circ \nu_F$ which satisfies the following definition. 

\begin{df}\label{df1}
A numbered group $(G,\nu)$ is \textbf{computably enumerable} if the set 
$$
\mathsf{MultT} := \{(i,j,k):\quad \nu(i)\nu(j)=\nu(k)\}$$ 
is computably enumerable.
\end{df} 

\begin{rem}\label{rp}
Let $(G,\nu)$ be a computably enumerable group.

\begin{enumerate}[(i)]
\item There exists a computable function $\star:\mathbb{N}\times \mathbb{N} \rightarrow \mathbb{N}$ such that for all $ x,y \in \mathbb{N}$ the equality $\nu(x)\nu(y) = \nu(x\star y)$ holds. 
\item There is a computable function which for every $x\in \mathbb{N}$  finds $y\in \mathbb{N}$ with $\nu(x)\nu(y)=1$, i.e. 
a number $(\nu(x))^{-1}$. 
We denote it by $x^{*}$.
\item The sets $\{n:\nu(n)=1\}$ and $\{(n_1,n_2): \nu(n_1)=\nu(n_2)\}$ are computably enumerable.
\end{enumerate} 

\end{rem}

\begin{rem} 
 If $(G, \nu )$ is a numbered group and the set $\mathsf{MultT}$ 
from Definition \ref{df1} is computable, then 
\begin{enumerate}[(i)] 
\item $G$ has a computable presentation (possibly under another numbering). 
Indeed, in this case the set of the smallest numbers of the elements of $G$ is computable. 
Enumerating this set by natural numbers we obtain 
a required  $1$-$1$-enumeration. 
\item In this case we also have that the set $\{(n_1,n_2): \nu(n_1)=\nu(n_2)\}$ is computable. 
\end{enumerate} 

\end{rem}

Groups $(G, \nu )$ as in this remark are called {\bf computable} groups. 
They correspond to groups with solvable word problem.
In this case the numbering $\nu$ is called a {\bf constructivization}.

\subsection{Amenability}

Let $G$ be a group, and $D\subset\subset G$.
Given $n\in \mathbb{N}$, we say that a subset $F\subset\subset G$ is an $n$-\textbf{F\o lner set} with respect to $D$ if 
\begin{equation}
\forall x\in D\hspace{0.5cm}  \frac{|F\setminus xF|}{|F|} \leq \frac{1}{n}
\end{equation}   
\newline
We denote by $\mathfrak{F}$\o$l_{G, D}(n)$ the set of all $n$-F\o lner sets with respect to $D$.
Moreover, we say that a sequence $(F_j)_{j\in \mathbb{N}}$ 
of non-empty finite subsets of $G$ is a \textbf{F\o lner sequence} 
if for every $g \in G$ the following condition holds:
\begin{equation}
 \lim\limits_{j\rightarrow \infty}\frac{|F_j\setminus gF_j|}{|F_j|}=0.
\end{equation} 
We call the binary function:
\begin{equation}
F\o l_{G}(n,D) = min\{|F|: F\subseteq G \text{ such that } F \in \mathfrak{F} \o l_{G, D}(n)\},
\end{equation}
\newline
where the variable $D$ corresponds to finite sets, the \textbf{F\o lner function of} $G$, \cite{ver}.    

It is easy to see that existence of F\o lner sets for every $D$ and all $n$ 
is equivalent to existence of a F\o lner sequence, i.e.  
$G$ admits a F\o lner sequence if and only if 
$F\o l_{G}(n,D)< \infty$ for all finite $D\subset G$ and $n\in\mathbb{N}$.  
In fact this is the \textbf{F\o lner condition of amenability}.

\begin{df}
A summable non-zero function $h: G \rightarrow \mathbb{R}_{+}$, $||h||_{1,G}<\infty$, is $n$-\textbf{invariant} with respect to $D$, if 
\begin{equation}\label{rse}
\forall x\in D\hspace{0.5cm} \frac{||h - _{x}h||_{1,G}}{||h||_{1,G}}< \frac{1}{n},
\end{equation}
where $_{x}h(g):=h(x^{-1}g)$.
\newline
We denote by $\mathfrak{R}eit_{G, D}(n)$ the set of all summable non-zero functions from $G$ to $\mathbb{R}_{+}$, which are $n$-invariant with respect to $D$. 
\end{df}

\noindent The following facts are well known and/or easy to prove. 

\begin{lm}\label{fr} Let $F,D\subset\subset G$.

\begin{enumerate}[(i)]
\item $F \in \mathfrak{F} \o l_{G, D}(n) \implies \forall g \in G \quad Fg \in F \o l_{G, D}(n)$

\item $F \in \mathfrak{F} \o l_{G, D}(n) \iff \forall x \in D  \quad \frac{|F\cap xF|}{|F|}> 1 - \frac{1}{n}$

\item $F \in \mathfrak{F} \o l_{G, D}(2n) \iff \chi_F \in \mathfrak{R}eit_{G, D}(n)$

\item If $h\in \mathfrak{R}eit_{G,D}(n)$ has a finite support then there exists $F \subset Supp(h)$ such that for all $x \in D$ following holds: $$ \frac{|F\setminus xF|}{|F|}< \frac{|D|}{2n}.$$

\end{enumerate}
\end{lm}

\subsection{Effective amenability}

In this section $(G,\nu)$ is a numbered group.

\begin{df}\label{cr} 
We say that $(G,\nu)$ has \textbf{computable Reiter functions}, if there exists an algorithm which, for every $n\in \mathbb{N}$ and any finite set $D\subset\mathbb{N}$ finds $f: \mathbb{N} \rightarrow \mathbb{Q_{+}}$, such that $|Supp(f)| < \infty$ and 
$$
\forall x \in D, \quad \frac{||\nu_{G *}(f) - _{\nu (x)}\nu_{G *}(f)||_{1,G}}{||\nu_{G *}(f)||_{1,G}}< \frac{1}{n}, 
$$
where $\nu_{G *}(f)(g):=\sum\limits_{i\in\nu^{-1}(g)}f(i)$.
\end{df}

In the case of the F\o lner condition of amenability, we consider three types of effectiveness.

\begin{df}\label{cea} 
The group $(G,\nu)$ is \textbf{$\Sigma$-amenable} if
there exists an algorithm which for all pairs $(n,D)$, where $n\in \mathbb{N}$ and 
$D\subset\subset\mathbb{N}$, finds a set $F\subset\subset\mathbb{N}$ containing 
a subset $ F'$, such that $ \nu(F') \in \mathfrak{F} \o l_{G, \nu(D)}(n)$.
\end{df}

\begin{df}\label{cf}
We say that $(G,\nu)$ has \textbf{computable F\o lner sets} 
if there exists an algorithm which, for all pairs $(n,D)$, 
where $n\in \mathbb{N}$ and $D\subset\subset\mathbb{N}$, 
finds a finite set $F\subset\mathbb{N}$ such that 
$\nu(F) \in \mathfrak{F}$\o$l_{G, \nu(D)}(n)$.
\end{df}

\begin{df}\label{ca} 
The group $(G,\nu)$ is \textbf{computably amenable} if
there exists an algorithm which for all pairs $(n,D)$, 
where $n\in \mathbb{N}$ and $D\subset\subset\mathbb{N}$, finds a set $F\subset\subset\mathbb{N}$ such that $\nu(F) \in \mathfrak{F} \o l_{G, \nu(D)}(n)$ and $|F|=|\nu(F)|$.
\end{df}

\section{Effective amenability of computably enumerable groups}

The main result of this section, Theorem \ref{re}, is a natural generalization of a theorem of 
M. Cavaleri from \cite{MC3} (Theorem 3.1) to the case of groups which are not finitely generated. 
In fact we use the same arguments (appropriately adapted to our case).
Throughout this section we assume that $(G,\nu)$ is a computably enumerable group.

We start with some preliminary material concerning Reiter functions and partitions. 
Let $X$ be a nonempty set. 
The family of sets $\Pi$ is a \textbf{partition of a set} $X$, if and only if all of the following conditions hold:
\begin{enumerate}
\item $\emptyset\notin \Pi$;
\item $\bigcup\limits_{A\in\Pi}A=X$;
\item $\forall A,B\in \Pi, A\neq B \implies A\cap B = \emptyset$.
\end{enumerate}
Partition $\Pi'$ is finer than partition $\Pi$ (denoted by $\Pi' \leq \Pi$), if for all $A'\in\Pi'$ there exists $A\in\Pi\text{, such that } A'\subset A$.

Let $f:\mathbb{N}\rightarrow\mathbb{Q}_{+}$ be a function with finite support $F$. 
Let $D$ be a finite subset of $\mathbb{N}$. 
With every partition $\Pi$ of the set $F$ and every $x\in D$ we associate the positive rational number:
$$
M_{\Pi}^{x}(f):= \frac{\sum_{V\in\Pi}|\sum_{v\in V}(f(v)-f(x^{*}\star v))|}{\sum_{v\in F}f(v)} , 
$$ 
where functions $\star$ and $^{*}$ are taken from Remark \ref{rp}. 
We denote by  $P$ the canonical partition of the set $F$, 
i.e. the partition into sets $ \{\nu^{-1}(\nu(k)), k\in F\}\cap F$. 
Then for every $x\in D$ we have 
\begin{equation}\label{5}
M_{P}^{x}(f) = \frac{||\nu_{G *}(f)-_{\nu(x)}\nu_{G *}(f)||_{1,G}}{||\nu_{G *}(f)||_{1,G}}.
\end{equation}
By the triangle inequality for any two partitions $\Pi, \Pi'$ of set $F$, $\Pi \leq \Pi'$ implies $M_{\Pi}^{x}(f) \geq M_{\Pi'}^{x}(f)$. 
In particular, for any partition  $\Pi \leq P$ and any $x\in D$ the following inequality holds:
\begin{equation}\label{6}
M_{\Pi}^{x}(f) \geq M_{P}^{x}(f).
\end{equation}

\begin{lm}\label{rr}

Let $(G,\nu)$ be a computably enumerable group. 
There exists a computable enumeration of the set of all triples $(n, D, f)$, where 
$D \subset\subset\mathbb{N}$ and $f:\mathbb{N}\rightarrow\mathbb{Q}_{+}$ is a finitely supported function, 
such that $\nu_{G *}(f) \in \mathfrak{R}eit_{G,\nu(D)}(n)$.

\end{lm}

\begin{proof}
We apply the method of Theorem 3.1($(i)\rightarrow (iv)$) of \cite{MC3}. 
Let us fix an enumeration of functions $f_i$ with finite support and 
the corresponding enumeration of all triples of the form $(n_i,D_j,f_k)$. 
The following procedure, denoted below by $\kappa(n,D,f)$, determines triples satisfying the condition of the lemma.

We define the algorithm $\kappa(n,D,f)$ as follows. 
For an input $f$ let $F=supp f$ and $P_0 := \{\{x\}:x\in F\}$, i.e. the finest partition of $F$. 
Let us fix an enumeration of the set $\{(n_1,n_2): \nu(n_1)=\nu(n_2)\}$.
Then on the $m$-th step of this enumeration we are trying to merge elements of the partition $P_{m-1}$ obtained at step $m-1$. 
We do so when we meet $(n_1,n_2)$, such that 
$|V_i\cap \{n_1,n_2\}|=|V_j \cap \{n_1,n_2\}|=1$  for some pair $V_i, V_j \in P_{m-1}$. 
In this case we just merge this pair. 
We see that $P_m \leq P$. 
Then we verify if $M_{P_{m}}^{x}(f) \leq \frac{1}{n}$ for all $x\in D$. 
We stop when these inequalities hold or when  $P_m = P$. 
In the former case  by (\ref{5}) and (\ref{6}) the function $\nu_{G *}(f)$ is $n$-invariant.
If there exist $x$, such that $M_{P_{m}}^{x}(f) > \frac{1}{n}$ and $P_m = P$, 
 then the function $\nu_{G *}(f)$ is not $n$-invariant. 

\end{proof}

The following theorem is a part of Therem \ref{1} from the introduction.

\begin{thm}\label{re}
Let $(G,\nu)$ be a computably enumerable group. 
Then the following conditions are equivalent:

\begin{enumerate}[(i)]
\item $(G,\nu)$ is amenable; 

\item $(G,\nu)$ has a subrecursive F\o lner function;

\item $(G,\nu)$ is $\Sigma$-amenable;

\item $(G,\nu)$ has  computable Reiter functions.

\end{enumerate}
\end{thm}

\begin{proof}
It is clear that 
(iii)$\implies$(ii)$\implies$(i). 

(iv)$\implies$(iii). 
By Definition \ref{cr} for all $n\in \mathbb{N}$ and every $D \subset\subset\mathbb{N}$ we find a function 
$f: \mathbb{N} \rightarrow \mathbb{Q^{+}}, |supp(f)| < \infty$, such that 
$\nu_{G *}(f)\in \mathfrak{R}eit_{G, D}$. 
Denote $F := supp(f)$. By Lemma \ref{fr} $(iv)$, there exists $\epsilon\in \mathbb{R}^{+}$ 
such that $\{g\in G: \nu_{G *}(f)(g) > \epsilon\}$ contains a subset that belongs to $\mathfrak{F}\o l_{G, \nu(D)}(n)$.
Since $\{g\in G: \nu_{G *}(f)(g) > \epsilon\} \subset \nu(F)$, then there exists 
$F'\subseteq F$ such that $\nu(F')$ satisfies the F\o lner condition. 

To prove (i)$\implies$(iv) let us assume that the group $G$ is amenable. 
Therefore for any $n$ and $D$ there exists $F\subset\subset\mathbb{N}$ such that 
$\nu(F) \in \mathfrak{F}\o l_{G, \nu(D)}(2n)$ and $|F| = |\nu(F)|$. 
Since $\nu$ is injective on $F$, $\nu_{G *}(\chi_F)=\chi_{\nu(F)}\in \mathfrak{R}eit_{G, D}(n)$. 
We fix an enumeration of finite subsets of $\mathbb{N}: F_1,F_2,\ldots$ and we start the algorithms 
$\kappa(n,D,\chi_{F_1}),\kappa(n,D,\chi_{F_2}),\ldots$ constructed in Lemma \ref{rr}, 
until one of them stops giving us a Reiter function for $\nu(D)$.

\end{proof}

\section{Effective amenability of computable groups}

The main results of this section, correspond to Theorem 4.1 and Corollary 4.2 of M. Cavaleri  from \cite{MC3}. 
In the proof we will use functions $\star$ and $^{*}$ from Remark \ref{rp}. 

\begin{thm}\label{ce}
Let $(G,\nu)$ be a computably enumerable group. The following conditions are equivalent:

\begin{enumerate}[(i)]
\item $(G,\nu)$ is amenable and computable; 

\item $(G,\nu)$ is computably amenable (Definition \ref{ca}).

\end{enumerate}
\end{thm}

\begin{proof}
(i)$\implies$(ii).
Suppose that $(G,\nu)$ is amenable and computable. 
Let $D\subset\subset \mathbb{N}$. According the enumeration of all finite sets for every $F\subset\subset \mathbb{N}$ 
we verify if the conditions of (ii) are satisfied. 
Verifying all equalities of the form $\nu(f_i)\nu(d_k)=\nu(f_j)$, where $f_i,f_j\in F$ and $d_k\in D$, 
we can algorithmically check if $\nu(F)\in \mathfrak{F} \o l_{G, \nu(D)}(n)$.
Verifying all equalities of the form $\nu(f_k)=\nu(f_l)$, where $f_k, f_l \in F$, we can check if $|F|=|\nu(F)|$. 
Since $(G,\nu)$ is amenable we eventually find the required $F$.

(ii)$\implies$(i).
Our proof is a modification of the construction of Theorem 4.1 from \cite{MC3}.
It is clear that the existence of an algorithm for (ii) implies amenability of $(G,\nu)$. 
Therefore we only need to show that $(G,\nu)$ is computable.
It is sufficent to show that for any $n_1,n_2,n_3 \in \mathbb{N}$ we can check if $\nu(n_1)\nu(n_2)=\nu(n_3)$.

Fix $n_1,n_2,n_3$. 
Let $D$ be the set $\{n_{1},n_{2},n_{3}\}$.
We use the algorithm for (ii) to find a set $F$ corresponding to $4$ and $D$, 
i.e. $\nu(F) \in F \o l_{G, \nu(D)}(4)$ and $|F|=|\nu(F)|$. Let $F = \{f_1, f_2,\ldots, f_k\}$. 

We fix an enumeration of the set of triples $\mathsf{MultT}$. 
Using it we will enumerate the (directed) graph of the action by multiplication of 
$\nu(n_1),\nu(n_2)$ and $\nu(n_3)$  on the set $\nu(F)$. 
We start by setting $\Sigma_{1}^{0} = \Sigma_{2}^{0} = \Sigma_{3}^{0} = \emptyset$.
At the m-th step of the construction we verify if the $m$-th triple of $\mathsf{MultT}$ 
is a triple of the form $n_l\star f_i = f_j$ for $l=1,2,3$.  
In this case we extend the corresponding 
$\Sigma_{l}^{m-1}$ by the pair  $(i,j)$.  
The graphs after step $m$ are denoted by $\Sigma_{l}^{m}$,  $l=1,2,3$.  

Next we verify $\min\limits_l |\Sigma_{l}^{m}|>\frac{3k}{4}$. 
If the inequality holds we stop the construction with $\Sigma_{l}:=\Sigma_{l}^{m}$.
\newline 
Since
 $$
\frac{|\{(i,j): \nu(n_l\star f_i\star f_j^{*})=1\}|}{k}\geq \frac{|\nu(F)\cap\nu(n_l)\nu(F)|}{|\nu(F)|}>\frac{3}{4}
$$
the procedure stops at some step $m$.

Let 
$$
\Sigma = \{i\in [k]: \exists j_1,j_2\in [k], (i,j_1)\in\Sigma_1, (j_1,j_2)\in\Sigma_2, (i,j_2)\in\Sigma_3\}. 
$$
If $\nu(n_1)\nu(n_2)=\nu(n_3)$ then for all $i\in [k]$, $\nu(n_1)\nu(n_2)\nu(f_i)=\nu(n_3)\nu(f_i)$. 
Since each of the partial permutations $\Sigma_1, \Sigma_2,\Sigma_3$ can be undefined for at most 
$\frac{1}{4}|F|$ elements from $|F|$, then $\nu(n_1)\nu(n_2)=\nu(n_3)$ implies $|\Sigma|\geq\frac{1}{4}|F|$.  
If equality does not hold, then $\Sigma$ is an empty set. 
When we see which possibility holds we decide if   $\nu(n_1)\nu(n_2)=\nu(n_3)$. 

\end{proof}

The proof of Theorem \ref{ce} gives the following interesting observation. 

\begin{cor} 
Let $(G,\nu)$ be a computably enumerable, amenable group. 
If for some $n\geq 4$ there exists an algorithm, which for every $D\subset\subset\mathbb{N}$ finds a set $F\subset\subset\mathbb{N}$ such that $\nu(F) \in F \o l_{G, \nu(D)}(n)$ and $|F|=|\nu(F)|$, then $G$ is computable.
\end{cor}
\noindent Using Theorem \ref{ce} we deduce a version of Theorem \ref{re} 
for computable groups. 
This finishes the proof of Theorem \ref{1}. 

\begin{thm}\label{eq}
Let $(G,\nu)$ be a computable group. Then the following conditions are equivalent:

\begin{enumerate}[(i)]
\item $(G,\nu)$ is amenable; 

\item $(G,\nu)$ is computably amenable;

\item $(G,\nu)$ has computable F\o lner sets;

\item $(G,\nu)$ has computable Reiter functions;

\item $(G,\nu)$ has subrecursive F\o lner function. 

\end{enumerate}
\end{thm}

\begin{proof}
By Theorem \ref{ce} we have (i)$\Rightarrow$(ii) and by Lemma \ref{fr}(iv) we have (iv)$\Rightarrow$(iii). Both 
(ii)$\Rightarrow$(iii)$\Rightarrow$(i) and (ii)$\Rightarrow$(v)$\Rightarrow$(i) are easy to see.

It follows that we only need to show that (ii)$\Rightarrow$(iv). 
We start with a finite set $D$ and use an algorithm of (ii) to find a set $F$ corresponding to $2n$. 
Then the characteristic function $\chi_F$ can be taken as $f$ from Definition \ref{cr}. 
Indeed since the function $\nu$ is injective on $F$ then $\nu_{G *}(\chi _{F})$ is the characteristic function of $\nu(F)$, 
which is $n$-invariant by Lemma \ref{fr}(iii).
\end{proof}

\section{Effective F\o lner sequence}

Let $(G,\nu)$ be a computable group. 
Since in the case of computable groups we can assume that function $\nu$ is $1$-$1$, we identify the set $G$ with $\mathbb{N}$ and subsets $F$ of $N$ with $\nu(F)\subset G$.

The \textbf{effective F\o lner sequence} of the group $(G,\nu)$, is an effective sequence $(n_j)_{j\in\mathbb{N}}$ such that for each $j$, $n_j$ is a G\"odel number of the set $F_j$, with $(F_j)_{j\in\mathbb{N}}$ being a F\o lner sequence. 

In the previous section we have shown that amenability of $(G,\nu)$ is equivalent to computable amenability. 
Note that this is also equivalent to existence of effective F\o lner sequences. 
Indeed, given $j$ we use the algorithm for computable amenability and compute the G\"odel number 
$n_j$ of some $F_j \in \mathfrak{F} \o l_{G, [j]}(j)$. 
Clearly, the sequence $(F_j)_{j\in\mathbb{N}}$ is a F\o lner sequence and a sequence 
$(n_j)_{j\in\mathbb{N}}$ is an effective F\o lner sequence. 

The following Theorem classifies the set of all effective F\o lner sequences of the group $(G,\nu)$ 
in the Arithmetical Hierarchy.
The idea of it belongs to Aleksander Ivanov. 

\begin{thm}\label{fs} 
Let $(G,\nu)$ be a computable group. 
The set of all effective F\o lner sequences of $(G,\nu)$ belongs to the class $\Pi^0_3$. 
Moreover, for $G=\bigoplus\limits_{n\in\omega} \mathbb{Z}$ it is a $\Pi^0_3$-complete set.
\end{thm}

\begin{proof} 
Let $\varphi(x,y)$ be a universal recursive function, and $\varphi_x(y)=\varphi(x,y)$ be a recursive function with a number $x$. 
We identify effective F\o lner sequences with numbers of recursive functions which produce these sequences.
The set of these numbers is denoted by $\mathfrak{F}_{seq}(G)$.
Then $m$ is a number of an effective F\o lner sequence if and only if the following formula holds:
\begin{align}
(\phi(m,y) &\text{ is a total function})\wedge\notag
(\forall g\in G)(\forall n)(\exists l)(\forall k)\Big(k>l\wedge (\phi(m,k)=f)\\ 
&\wedge(\mbox{$f$ is a G\" odel number of $F_j$})\rightarrow \frac{|F_j\setminus gF_j|}{|F_j|}<\frac{1}{n}\Big).\label{seq}
\end{align}

Given number $f$ the inequality $\frac{|F_j\setminus gF_j|}{|F_j|}<\frac{1}{n}$ can be verified effectively. 
Since the set of numbers of all total functions belongs to the class $\Sigma_2^0$ it is easy to see that the set of all $m$ which satisfy (\ref{seq}) is a $\Pi_3^0$ set. This proves the first part of the theorem.

We remind the reader that $W_x=Dom\varphi_x$ is the computably enumerable set with a number $x$.
The set $\overline{Cof}=\{e:\forall n\; W_{\varphi_e(n)} \text{ is finite}\}$, is known to be a $\Pi^0_3$-complete set (\cite{sri}, p. 87). 
To prove the second part of the theorem, assume that $G=\bigoplus\limits_{n\in\omega} \mathbb{Z}$. Let us show that the set $\overline{Cof}$ is reducible to $\mathfrak{F}_{seq}(G)$.
For each $e$ let us fix a computable enumeration of the set $\{(n,x): x\in W_{\varphi_e(n)}\}$.
We can assume that this enumeration is without repetitions.

We present $\bigoplus\limits_{n\in\omega} \mathbb{Z}$ as $\bigoplus\limits_{n\in\omega} \langle g_n\rangle$. We shall construct a sequence $\{F_s^e\}$ such that $e\in\overline{Cof}$ iff $\{F_s^e\}$ is a F\o lner sequence. 

For a given $s$, we use the enumeration of the set $\{(n,x): x\in W_{\varphi_e(n)}\}$ to find the element $(n_s,x)$ with the number $s$. For each $i=1,\ldots, s$ such that $i\neq n_s$ let $F_{s,i}=\{g_{i},g_{i}^2\ldots g_{i}^s\}$. For $i=n_s$ we put $F_{s,i}=\{g_{i}\}$. Let $F_s^e=\bigoplus\limits_{1}^{s} F_{s,i}$. Then in the former case  $F_s^e$ is an $s$-F\o lner set with respect to $g_i$ and in the latter case $F_s^e$ is not a $2$-F\o lner set with respect to $g_i$. This ends the construction.

\noindent \textit{Case 1.} $e\notin\overline{Cof}$. There exists $n'$ such that $W_{\varphi_{e(n')}}$ is an infinite set. Therefore there exist an increasing sequence $\{s_i\}$ and the number $i'$ such that for all $i>i'$, $F_{s_i}^e$ is not a $2$-F\o lner set with respect to $g_{n'}$. Clearly the number of a sequence $\{F_s^e\}$ does not belong to the set of numbers of a F\o lner sequences.

\noindent \textit{Case 2.} $e\in\overline{Cof}$. For all $n$, $W_{\varphi_{e(n)}}$ is a finite set. Therefore for all $n$, there exists the number $s'$ such that for all $s>s'$, $F_{s}^e$ is an $s$-F\o lner set with respect to $g_{n}$.  This sequence is a F\o lner sequence.

Since for every $e$ the number of the algorithm producing $\{F^e_s\}$ can be effectively found it follows that the set $\overline{Cof}$ is reducible to $\mathfrak{F}_{seq}(G)$, which completes the proof.

\end{proof}

\section{An effective version of Hall's Harem Theorem}
In this section we generalize the work of Kierstead  \cite{hak} 
concerning an effective version of the Hall's Theorem. 
These results will be applied in the next section 
to effective paradoxical decompositions. 
Below we follow the presentation of \cite{hak}.

A graph $\Gamma=(V,E)$ is called a \textbf{bipartite graph} if the set of vertices $V$ is partitioned into sets $A$ and $B$ in such way, that the set of edges $E$ is a subset of $A\times B$. We denote such a bipartite graph by $\Gamma=(A,B,E)$. The set $A$ (resp. $B$) is called the set of \textbf{left} (resp. \textbf{right}) \textbf{vertices}.

From now on we concentrate on bipartite graphs. 
Although our definitions concern this case they usually have obvious extensions to all ordinary graphs. Let $\Gamma=(A,B,E)$. We will say that an edge $(a,b)$ is \textbf{adjacent} to vertices $a$ and $b$. In this case we say that $a$ and $b$ are adjacent. We also say that  two edges $(a,b),(a',b')\in E$ are \textbf{adjacent} if they have a common adjacent vertex.

Given a vertex $x\in A\cup B$ the \textbf{neighbourhood} of  $x$ is a set $$N _{\Gamma}(x)=\{y\in A\cup B: (x,y)\in E\}.$$  
For subsets $X\subset A$ and $Y\subset B$, we define the neighbourhood 
$N _{\Gamma}(X)$ of $X$ and the neighbourhood $N _{\Gamma}(Y)$ of $Y$ by

$$
N _{\Gamma}(X)=\bigcup\limits_{x\in X} N _{\Gamma}(x) \text{ and } N _{\Gamma}(Y)=\bigcup\limits_{y\in Y} N _{\Gamma}(y).
$$ 
We drop the subscript $\Gamma$ if it is clear from the context.

The subset $X$ of $A$ (resp. $Y$ of $B$) is called \textbf{connected} if for all $x, x' \in X$ (resp. $y, y' \in Y$) there exist a path $x=p_0,p_1,\ldots, p_k=x'$ in $\Gamma$ such that for all $i$ $p_i\in X\cup N_{\Gamma}(X)$.

We say that $\Gamma$ is \textbf{locally finite} if the set $N(x)$ is finite for all $x\in A\cup B$. 
If $\Gamma$ is locally finite then the sets $N(X)$ and $N(Y)$ are finite for all finite subsets $X\subset A$ and $Y\subset B$.

For a given vertex $v$ a \textbf{star} of $v$ is a subgraph  
$S=(V',E')$ of $\Gamma$, with $V'=\{v\}\cup N_{\Gamma}(v)$ and $E'=\{(v,v')\in E\}$.

A \textbf{matching} (a $(1,1)$-\textbf{matching}) from $A$ to $B$ 
is a subset $M\subset E$ of pairwise nonadjacent edges. 
A matching $M$ is called \textbf{left-perfect} (resp. \textbf{right-perfect}) 
if for all $a \in A$ (resp. $b\in B$)  there exists exactly one 
$b\in B$ (resp. $a\in A$) with $(a,b)\in M$.
The matching $M$ is called \textbf{perfect} if it is both right and left-perfect.

We now introduce perfect $(1,k)$-matchings from $A$ to $B$ without defining $(1,k)$-matchings. 
We will use only perfect ones.

\begin{df}
A \textbf{perfect} $(1,k)$-\textbf{matching} from $A$ to $B$ is a set $M\subset E$ satisfying following conditions: 

\begin{enumerate}[(1)]
\item for all $a \in A$ there exists exactly $k$ vertices $b_1,\ldots b_k \in B$ such that $(a,b_1),\ldots,(a,b_k)\in M$;

\item for all $b \in B$ there is an unique vertex $a\in A$ such that $(a,b)\in M$.

\end{enumerate} 
\end{df}

The following Theorem is known as \textbf{the Hall's Harem Theorem}, 
and the first of equivalent conditions is known as \textbf{Hall's $k$-harem condition}.

\begin{thm}
Let $\Gamma=(A,B,E)$ be a locally finite graph and let $k\in \mathbb{N},\; k\geq 1$. The following conditions are equivalent:

\begin{enumerate}[(i)]

\item For all finite subsets $X\subset A$, $Y\subset B$ following inequalities holds $|N(X)|\geq k|X|$, $|N(Y)|\geq \frac{1}{k}|Y|$.

\item $\Gamma$ has a perfect $(1,k)$-matching.

\end{enumerate} 
\end{thm}

Given a $(1,k)$-matching $M$ and a vertex $a\in A$ an $M$-star of $a$ is a graph consisting of the set of all vertices and edge adjacent to $a$ in $M$.

\begin{df} A graph $\Gamma$ is \textbf{computable} if there exists a bijective 
function $\nu: \mathbb{N}\rightarrow V$ such that the set 
$$
R:=\{(i,j): (\nu(i),\nu(j))\in E\} 
$$
is computable. 
A locally finite graph $\Gamma$ is called \textbf{highly computable} if additionally there is a recursive function $f: \mathbb{N}\rightarrow \mathbb{N}$ such that 
$f(n)=|N_{\Gamma}(\nu(n))|$ for all $n\in\mathbb{N}$. 
This definition and the three definitions below are due to Kierstead \cite{hak}.
\end{df}

\begin{df}
A bipartite graph $\Gamma=(A,B,E)$ is \textbf{computably bipartite} 
if $\Gamma$ is computable and the set of $\nu$-numbers of $A$ is computable.
\end{df}
Below we will identify the elements of $\Gamma$ with numbers.

\begin{df}\label{cpkm}
Let $\Gamma=(A,B,E)$ be a computably bipartite graph. 
A perfect $(1,k)$-matching $M$ from $A$ to $B$ is called a \textbf{computable perfect} $(1,k)$-\textbf{matching} if there is an algorithm which
\begin{itemize}
\item for each $i$ with $\nu(i)\in A$, finds the tuple $(i_1,i_2, \ldots, i_k)$ such that $(\nu(i),\nu(i_j))\in M$, for all \break $j=1,2,\ldots, k$
\item when $\nu(i)\notin A$ it finds $i'$ such that $(\nu(i'),\nu(i))\in M$.
\end{itemize}

\end{df}

The remainder of this section will be devoted to a proof that the following condition 
implies the existence of the computable perfect $(1,k)$-matching.

\begin{df} 
A bipartite graph $\Gamma=(A,B,E)$ satisfies 
the \textbf{computable expanding Hall's harem condition with respect to $k$} 
(denoted $c.e.H.h.c.(k)$), if and only if there is a recursive function 
$h: \mathbb{N} \rightarrow \mathbb{N}$ such that:
\begin{itemize}
\item $h(0)=0$
\item for all finite sets $X\subset A$, the inequality $h(n)\leq |X|$ implies $n\leq |N(X)|-k|X|$
\item for all finite sets  $Y\subset B$, the inequality $h(n)\leq |Y|$ implies $n\leq |N(Y)|-\frac{1}{k}|Y|$.
\end{itemize}
\end{df}

Clearly, if the graph $\Gamma$ satisfies the $c.e.H.h.c.(k)$, then it satisfies the Hall's $k$-harem condition. 

\begin{thm}
If $\Gamma=(A,B,E)$ is a highly computable bipartite graph satisfying 
the $c.e.H.h.c.(k)$, then $\Gamma$ has a computable perfect $(1,k)$-matching.
     
\end{thm} 

\begin{proof}
We extend the proof of Theorem 3 of the Kierstead's paper \cite{hak}. We fix a computable enumeration of $A$ and $B$.
Let $h$ witness the $c.e.H.h.c.(k)$ for $\Gamma$.
We begin by setting $M=\emptyset$.
At step $s$ we update already constructed $M$ in the following way. For a vertex $x_s\in A\cup B$ we construct some subgraph $\Gamma_s$ and a matching $M_s$ in $\Gamma_s$. The matching $M$ is updated by those elements of $M_s$ which contain $x_s$. The subgraph $\Gamma_s$ is constructed so that after removal of the $M_s$-star of $x_s$ from $\Gamma$, we still have a highly computable bipartite graph satisfying the $c.e.H.h.c.(k)$.
  
At the first step of the algorithm we choose $a_0$, the first element of the set $A$. 
We construct the induced subgraph $\Gamma_0=(A_0,B_0,E_0)$ so that $A_0\cup B_0$ is the set of vertices with distance of at most $\max\{2h(k)+1,3\}$ from $a_0$.
Since the graph $\Gamma$ is highly computable the graph $\Gamma_0$ is finite and can be found effectively. 
It is clear that for all vertices $v$ from $A_0$, $N_{\Gamma_{0}}(v)=N_{\Gamma}(v)$. 
Therefore, for all $X\subset\subset A_0$ the inequality $h(n)\leq |X|$ implies $n\leq |N_{\Gamma_{0}}(X)|-k|X|$.

Let $B_{S_0}$ denote the set of vertices $v\in B_0$ of the distance $\max\{2h(k)+1,3\}$ from $a_0$. 
It is clear that $N_{\Gamma_{0}}(B_0\setminus B_{S_0})=N_{\Gamma}(B_0\setminus B_{S_0})=A_0$. 
On the other hand since it can happen that $N_{\Gamma}(B_{S_0})$ is not contained in $A_0$, it is possible that there exist $Y\subset B_{S_0}$, such that $|N_{\Gamma_{0}}(Y)|\leq\frac{1}{k}|Y|$. 

Since $\Gamma$ contains a perfect $(1,k)$-matching, there exists a $(1,k)$-matching in $\Gamma_0$, that satisfies the conditions of perfect $(1,k)$-matchings for all $a\in A_0$, $b\in B_0 \setminus B_{S_0}$. 
We denote it by $M_0$. Since $\Gamma_0$ is finite, the matching $M_0$ can be obtained effectively. 
Let $\{b_{0,1},\ldots, b_{0,k}\}$ be all  elements that $(a_0,b_{0,i})$ belongs to $M_0$. 
We define $M$ to be the set of all these pairs.

Let $\Gamma'$ be a subgraph obtained from $\Gamma$ through removal of the $M_0$-star of $a_0$.
Since the sets $A\cup B$, $A$ and $E$ are computable, and the matching $M_0$ is found effectively, hence the sets $A'\cup B'$, $A'$ and $E'$ are also computable.
Therefore $\Gamma'$ is a computably bipartite graph.
Since $\Gamma'$ is locally finite and we can compute the neighbourhood of every vertex, $\Gamma'$ is highly computable.
To finish this step it suffices to show that $\Gamma'$ satisfies $c.e.H.h.c.(k)$.

Let 
$$
h'(n) = \left\{
\begin{array}{rr}
0, \quad \text{if} \quad n=0, \ \\
h(n+k), \quad \text{if} \quad n > 0.
\end{array}\right.
$$
We claim that $h'$ works for $\Gamma'$. We start with the case when $X\subset A'$ and $n>0$.
Since $|N_{\Gamma'}(X)|\geq |N_{\Gamma}(X)|-k$, then for $n\geq 1$ the inequality 
$|X|>h'(n)$ implies $|N_{\Gamma'}(X)|-k|X|\geq |N_{\Gamma}(X)|-k|X|-k\geq n$.

Let us consider the case when $n=0$ and $X$ is still a subset of $A'$. If $X$ is not connected, then its neighbourhood would be the union of nieghbourhoods of its connected subsets. Therefore without the loss of the generality, we can assume that $X$ is connected. If $X\subset A_0$, then $|N_{\Gamma'}(X)|- k|X| \geq 0$, since $M_0$ was a $(1,k)$-matching from $A_0$ to $B_0$ that was perfect for subsets of $A_0$. 

Now, let us assume that there exists $a'\in X\setminus A_0$.  
If $b_{0,1},\ldots, b_{0,k} \notin N_{\Gamma}(X)$, then $|N_{\Gamma'}(X)| = N_{\Gamma}(X)$, so $|N_{\Gamma'}(X)|-k|X|\geq 0$. Assume that for some $i\leq k$ and some $a\in X,$ there exists $(a,b_{0,i})\in E$. 
Since the distance between $a$ and $a'$ is at least $2h(k)$ we have $|X|\geq h(k)+1$. Thus $|N_{\Gamma}(X)|-k|X|\geq k$ and it follows that $|N_{\Gamma'}(X)|-k|X|\geq 0$. We conclude that the case of finite subsets of $A'$ is verified. 

Now we need to show that $\Gamma'$ satisfies 
$c.e.H.h.c.(k)$ for sets $Y\subset\subset B'$. 
We have to show that for all finite sets  $Y\subset B$, 
the inequality $h'(n)\leq |Y|$ implies $n\leq |N_{\Gamma'}(Y)|-\frac{1}{k}|Y|$.
Note $Y\subset\subset B'=B \setminus \{b_{0,1},\ldots, b_{0,k}\}$ 
and $|N_{\Gamma'}(Y)|\geq |N_{\Gamma}(Y)|-1$.

In the case $n\!>\!0$ the inequality $|Y|>h'(n)$ 
implies \!
$|N_{\Gamma'}(Y)|-\frac{1}{k}|Y|\geq |N_{\Gamma}(Y)|-\frac{1}{k}|Y|-1\geq n+k-1\geq n$.

Let us consider the case $n=0$. 
As before, we can assume that $Y$ is connected.
If $Y\subset B_0\setminus B_{S_0}$, then $|N_{\Gamma'}(Y)|- \frac{1}{k}|Y| \geq 0$, 
since $M_0$ satisfied the conditions of a perfect $(1,k)$-matching for elements of $B_0\setminus B_{S_0}$. 

Let us assume that there exists $b'\in Y\setminus (B_0\setminus B_{S_0})$. 
If $a_0\notin N_{\Gamma}(Y)$, then $N_{\Gamma'}(Y)=N_{\Gamma}(Y)$ and $|N_{\Gamma'}(Y)|- \frac{1}{k}|Y| \geq 0$. 

Assume that for some $b\in Y $ there exists the edge $(a_0,b)\in E$.
Since the distance between $b$ and $b'$ is at least $2h(k)$ we have $|Y|\geq h(k)+1$. It follows that $|N_{\Gamma}(Y)|- \frac{1}{k}|Y| \geq k$ and $|N_{\Gamma'}(X)|-\frac{1}{k}|X|\geq k-1\geq 0$.

As a result we have that the graph $\Gamma'$ satisfies $c.e.H.h.c.(k)$. To force the matching $M$ to be a perfect $(1,k)$-matching we use back and forth. Therefore we start the next step of an algorithm by choosing an element $b_{1,1}$ of $B'$.

We construct the induced subgraph $\Gamma_1=(A_1,B_1,E_1)$ so that $A_1\cup B_1$ is a set of vertices of $\Gamma'$ with distance of at most $\max\{2h'(k)+2,4\}$ from $b_{1,1}$.
Let $B_{S_1}$ denote the set of vertices of the distance $\max\{2h'(k)+2,4\}$ from $b_{1,1}$.
Since $\Gamma'$ contains a perfect $(1,k)$-matching, there exist a $(1,k)$-matching in $\Gamma_1$ that satisfies the conditions of a perfect $(1,k)$-matching for all $a\in A_1$ and $b\in B_1\setminus B_{S_1}$. We denote it by $M_1$. We choose $a_1$ with $(a_1,b_{1,1})\in M_1$. Let $\{b_{1,2},\ldots, b_{1,k}\}$ be all remaining elements that $(a_1,b_{1,i})$ belongs to $M_1$. We update $M$ by all edges adjacent to $a_1$ in $M_1$.

Let $\Gamma''$ be a subgraph obtained from $\Gamma'$ through removal of the $M_1$-star of $a_1$. Then $\Gamma''$ is also highly computable computably bipartite graph. We need to show that $\Gamma''$ satisifies $c.e.H.h.c.(k)$.

Let 
$$
h''(n) = \left\{
\begin{array}{rr}
0, \quad \text{if} \quad n=0, \ \\
h'(n+k), \quad \text{if} \quad n > 0.
\end{array}\right.
$$
To prove that $h''(n)$ works for $\Gamma''$ we use the same method as as in the case $h'(n)$ and $\Gamma'$.

We continue iteration by taking the elements of $A$ at even steps and 
the elements of $B$ at odd steps. 
At every step $n$, the graph $\Gamma^{(n)}$ satisfies the conditions for existence of perfect $(1,k)$-matchings and we update $M$ by $k$ edges adjacent to $a_n$. 
Every vertex $v$ will be added to $M$ at some step of the algorithm. 
It follows that $M$ is a perfect $(1,k)$-matching of the graph $\Gamma$. 
Effectiveness of our back and forth construction guarantees that we have an algorithm satisfying Definition \ref{cpkm}. 

\end{proof}

\section{Effective paradoxical decomposition}

Throughout this section, $(G,\nu)$ is a computable group. 
For simplicity of notation we identify the set $G$ with $\mathbb{N}$ and 
subsets $F$ of $\mathbb{N}$ with $\nu(F)\subset G$. 
As before by $x^{*}\in \mathbb{N}$ we denote a number with 
$\nu(x^* )\nu(x) =1$. 

\begin{df} \label{EPD} 
The group $G$ has an \textbf{effective paradoxical decomposition}, if there exists a finite set $K\subset G$ and two families of computable sets 
$(A_k)_{k\in K}, (B_k)_{k\in K}$, such that:
$$
G = \Big(\bigsqcup\limits_{k\in K}kA_k\Big)\bigsqcup\Big(\bigsqcup\limits_{k\in K}kB_k\Big)=\Big(\bigsqcup\limits_{k\in K}A_k\Big)=\Big(\bigsqcup\limits_{k\in K}B_k\Big).
$$
We call $(K,(A_k)_{k\in K}, (B_k)_{k\in K})$ a paradoxical decomposition of $G$.
\end{df}

\begin{thm} \label{EPDthm}
There is an effective procedure which for any finite subset $K_0 \subset G$ 
satisfying the condition: 
\begin{quote}
there is a natural number $n$ such that  for any finite subset $F\subset G$, 
there exists $k \in K_0$ such that $\frac{|F\setminus kF|}{|F|}\geq \frac{1}{n}$, 
\end{quote} 
finds a finite subset $K \subset G$ which defines 
an effective paradoxical decomposition as in Definition \ref{EPD}.
\end{thm}

\begin{proof}
The proof is an adaptation of the proof of Theorem 4.9.2 from \cite{csc}.
Consider the set $K_{1} = K_{0}\cup\{1\}$. 
For any $F\subset\subset G$ we have: 
$$
K_{1}F \supset F \text{ and } K_{1}F\setminus F=K_{0}F\setminus F. 
$$ 
Thus there is $k\in K_0$ so that  
$$
|K_{1}F|-|F|=|K_{1}F\setminus\ F|=|K_{0}F\setminus F|\geq |kF\setminus F|\geq \frac{|F|}{n}.
$$
It follows that

$$
|K_{1}F|\geq (1+\frac{1}{n})|F|.
$$

Choose $n_1\in \mathbb{N}$ such that $(1+\frac{1}{n})^{n_1}\geq 3$ and 
set $K=K_{1}^{n_1}$. 
We see that $K$ is found effectively by $K_0$. 
Note that for any $ F\subset\Gamma$ we have $|KF|\geq 3|F|$.

To find the corresponding effective paradoxical decomposition 
consider the bipartite graph $\Gamma_K(G)=(\mathbb{N}, \mathbb{N}, E)$, 
where the set $E\subset \mathbb{N} \times \mathbb{N}$ consists of all pairs $(g,h)$ 
with $h \in Kg$, where $g,h$ are viewed as elements of $G$. 
Since $G$ is computable and $K$ is finite, the graph $\Gamma_K(G)$ is computably bipartite. Since the degree of every vertex is equal to $|K|$, the graph is highly computable.

Let $F$ be a finite subset of the first copy of $G$. 
Then $|N_{\Gamma}(F)|=|KF|\geq 3|F|$. 
It follows that: 
$$
|N_{\Gamma}(F)|-2|F|\geq 3|F|-2|F|=|F|.
$$ 
Therefore for any $ n\in \mathbb{N}$ the inequality $n\leq |F|$ implies that $n\leq |N_{\Gamma}(F)|-2|F|$.

On the other hand, if we consider a finite set $F$ in the second copy of $G$, 
then any $k\in K$ satisfies $N_{\Gamma}(F)\supset k^{*}F$. 
Consequently:
$$
|N_{\Gamma}(F)|\geq |k^{*}F|=|F|\geq \frac{1}{2}|F|.
$$ 
Since the function $h(n)=2n$ is recursive, 
the graph $\Gamma_K(G)$ satisfies $c.e.H.h.c.(2)$ with respect to $h$. 
By virtue of the Effective Hall Harem Theorem, we deduce the existence of a computable perfect $(1,2)$-matching $M$ in $\Gamma_K(G)$. 
In other words, there is a computable surjective $(2\rightarrow 1)$-map 
$\phi : \mathbb{N} \rightarrow \mathbb{N}$ such that 
$n\phi(n)^{*} \in K$ for all $n\in \mathbb{N}$.   

We now define functions $\psi_1, \psi_2$ as follows: 
$$
\left\{
\begin{array}{r}
\psi_1(n)=\min(n_1,n_2) \\
\psi_2(n)=\max(n_1,n_2)
\end{array}\right. , \text{ where } \phi(n_1)=n=\phi(n_2), n_1\neq n_2.
$$
Since the function $\phi$ realizes a computable perfect $(1,2)$-matching, 
both $\psi_1$ and $\psi_2$ are recursive. 

Define $\theta_1(n):=\psi_1(n)n^{*}$, $\theta_2(n):=\psi_2(n)n^{*}$. 
Observe that $\theta_1$, $\theta_2$ are recursive and 
$\theta_1(n), \theta_2(n)\in K$ for all $n\in \mathbb{N}$. 

For each $k\in K$ define sets $A_k$ and $B_k$ in the following way:
$$
A_k=\{n\in\mathbb{N}: \theta_1(n)=k\},\; 
B_k=\{n\in\mathbb{N}: \theta_2(n)=k\}.
$$
It is clear that these sets are computable and
$$
G=\bigsqcup\limits_{k\in K}A_k=\bigsqcup\limits_{k\in K}B_k.
$$
For each $n\in A_k$, the value $\psi_1(n)$ is $k\cdot n$ under the group multiplication. Thus $\psi_1(\mathbb{N})=\bigsqcup\limits_{k\in K}kA_k$. 
Similarly we can show that $\psi_2(\mathbb{N})=\bigsqcup\limits_{k\in K}kB_k$. Since $\mathbb{N}=\psi_1(\mathbb{N})\bigsqcup\psi_2(\mathbb{N})$, we have
$$
G = \Big(\bigsqcup\limits_{k\in K}kA_k\Big)\bigsqcup\Big(\bigsqcup\limits_{k\in K}kB_k\Big).
$$
Therefore $(K,(A_k)_{k\in K}, (B_k)_{k\in K})$ is an effective paradoxical decomposition of the group $G$.
\end{proof}

\section{Complexity of paradoxical decompositions}

We preserve the assumption of Section 7. 

\begin{df}\label{elf}
Let 
$$\mathfrak{W}_{BT}=\left\{K: (K\subset\subset G)\land\exists n\in \mathbb{N} \; (\forall F \subset\subset G)(\exists k \in K)\left(\frac{|F\setminus kF|}{|F|}\geq \frac{1}{n}\right)\right\}.$$
We call this family \textbf{witnesses of the Banach-Tarski paradox}.
\end{df}

This term is justified by Theorem \ref{EPDthm} where $\mathfrak{W}_{BT}$ 
appears in the formulation.

\begin{pr}
For any computable group the family $\mathfrak{W}_{BT}$ 
belongs to the class $\Sigma^{0}_{2}$.
\end{pr}

\begin{proof}
Since group $G$ is computable, for any finite subsets  $K$, $F$ of $G$, 
and any $n\in\mathbb{N}$, we can effectively check if the inequality 
$\frac{|F\setminus kF|}{|F|}< \frac{1}{n}$ holds for all $k\in K$. 
Therefore, the set of triples $(n,K,F)$ such that 
$ \frac{|F\setminus kF|}{|F|} < \frac{1}{n}$ holds for all $k\in K$ is computably enumerable. 

Since the projection of this set to the first two coordinates is also computably enumerable, the set 
$$
\mathfrak{W}_{BT}'=\{(K,n): (\forall F \subset\subset \Gamma)(\exists k\in K)(\frac{|F\setminus kF|}{|F|}\geq \frac{1}{n})\}
$$ 
belongs to the class $\Pi_{1}^{0}$. 
The set $\mathfrak{W}_{BT}$ consists of $K$ such that there exists 
$n\in \mathbb{N}$ with $(K,n)\in \mathfrak{W}_{BT}'$.  
Thus $\mathfrak{W}_{BT}$ belongs to the class $\Sigma^0_2$.
\end{proof}
\bigskip

The following question has become principal for us. 
\begin{itemize} 
\item Are there natural examples with computable/non-computable  $\mathfrak{W}_{BT}$? 
\end{itemize}

In the appendix we give an example of a finitely presented group with decidable word problem and 
non-computable  $\mathfrak{W}_{BT}$.  
In the present section we give positive examples. 
The most natural ones are provided by the following theorem.

\begin{thm}\label{fg}
The family $\mathfrak{W}_{BT}$ is computable for any finitely generated free group.
\end{thm}

The proof of this theorem is based  on some  reformulation of witnessing. 
It belongs to M. Cavaleri. It simplifies our original argument. 

\begin{pr}\label{fif}

Let $G$ be a group and $K \subset\subset G$.  
Then $K \in \mathfrak{W}_{BT}$ if and only if $\langle K \rangle$ 
is a non-amenable subgroup of $G$. 
\end{pr}

\begin{proof} 
The necessity is obvious. 
Assume that $K \notin \mathfrak{W}_{BT}$. It follows that for every $n$ there exists set $F_n$ such that $F_n \in \mathfrak{F}$\o$l_{G, K}(n)$.
Set $n\in \mathbb{N}$. Let $m=n|K|$. We now follow a proof of Proposition 9.2.13 from \cite{cor} to show that there exists $t_0\in G$ such that the set $F_mt_0^{-1} \cap \langle K \rangle =\{ k\in \langle K \rangle: kt_0 \in F_m  \}$ is an $n$-F\o lner for $K$. 
Let $T\subset G$ be a complete set of representatives of the right cosets of $\langle K \rangle$ in $G$. Clearly, every $g\in G$ can be uniquely written in the form $g=ht$ with $h\in \langle K \rangle$ and $t\in T$. We then have:

\begin{equation}\label{el1}
|F_m|= \sum\limits_{t\in T}|
F_mt^{-1} \cap \langle K \rangle|
\end{equation}
For every $x\in K$, we have $xF_m= \bigsqcup\limits_{t\in T}(xF_mt^{-1} \cap \langle K \rangle)t$, hence:

$$
xF_m\setminus F_m = \bigsqcup\limits_{t\in T}((xF_mt^{-1} \cap \langle K \rangle)\setminus (F_mt^{-1} \cap \langle K \rangle))t.
$$
This gives us:

\begin{equation}\label{el2}
|xF_m\setminus F_m|= \sum\limits_{t\in T}|(xF_mt^{-1} \cap \langle K \rangle)\setminus (F_mt^{-1} \cap \langle K \rangle)|.
\end{equation}
Since for all $x\in K$,
$$
|xF_m\setminus F_m|\leq \frac{|F_m|}{m},
$$
using (\ref{el1}) and (\ref{el2}), we get

$$\sum\limits_{t\in T}|(KF_mt^{-1} \cap \langle K \rangle)\setminus (F_mt^{-1} \cap \langle K \rangle)|=\sum\limits_{t\in T}|\bigcup\limits_{x\in K} ((xF_mt^{-1} \cap \langle K \rangle)\setminus (F_mt^{-1} \cap \langle K \rangle))|\leq \frac{|K|}{m}\sum\limits_{t\in T}|
F_mt^{-1} \cap \langle K \rangle|$$
By the pigeonhole principle, there exists $t_0\in T$ such that the set  $|(KF_mt_0^{-1} \cap \langle K \rangle)\setminus (F_mt_0^{-1} \cap \langle K \rangle)|\leq \frac{1}{n}|F_mt_0^{-1} \cap \langle K \rangle|$. Clearly $F_mt_0^{-1} \cap \langle K \rangle$ is an $n$-F\o lner set with respect to $K$. Since $n$ was arbitrary, $\langle K \rangle$ is amenable, a contradiction.
\end{proof}

\bigskip

\begin{proof} {\em (Theorem \ref{fg}).} 
Let $\mathbb{F}$ be a finitely generated free group. 
Since it is computable, the equation $xy = yx$  can be effectively verified for every $x,y\in \mathbb{F}$. 
We will show that $K\in \mathfrak{W}_{BT}$ if and only 
if there exist $ x,y\in K$ such that $ xy\neq yx$. 
This will give the result. 

$(\Rightarrow)$
Let us assume that $ xy=yx$ for every $x,y\in K$. 
Since $\mathbb{F}$ is a free group, there exists $z\in\mathbb{F}$ such 
that all words from $K$ are powers of $z$. 
Since the subgroup $\langle z \rangle$ is cyclic, the subgroup $\langle K\rangle$ is amenable and 
for every $n$ there is a finite set $F$, which is an $n$-F\o lner with respect to $K$. 
Clearly $K\notin \mathfrak{W}_{BT}$. 

$(\Leftarrow)$
Let us assume that there exist $x,y\in K$ with $ xy\neq yx$. 
Then $x,y$ generate a free subgroup of $\mathbb{F}$ of rank $2$. 
By Proposition \ref{fif} there is a natural number $n$ 
such that $\mathfrak{F}$\o$l_{\mathbb{F}, \{x,y\}}(n)=\emptyset$. 
Thus $\mathfrak{F}$\o$l_{\mathbb{F}, K}(n)$ is also empty.

\end{proof}

We remind the reader that a group $G$ is called \textbf{fully residually free} 
if for any finite collection of nontrivial elements 
$g_1 ,\ldots,g_n\in G\setminus\{1\}$ there exists a homomorphism 
$\phi:G\rightarrow \mathbb{F}$ onto a free group $\mathbb{F}$ such that 
$\phi(g_1)\neq 1,\ldots,\phi(g_n)\neq 1$, \cite{kap}. 
The class of fully residually free groups as well as residually free groups 
has deserved a lot of attention mainly in connection with 
algorithmic and model-theoretic investigations in group theory, 
see for example \cite{KhM} and \cite{Sela}. 

\bigskip

\begin{thm}
The family $\mathfrak{W}_{BT}$ is computable for any computable 
fully residually free group.
\end{thm}

\begin{proof}
Let $(G,\nu)$ be a computable fully residually free group. 
Since $(G,\nu)$ is computable,  
it suffices to show that $K\in \mathfrak{W}_{BT}$ if and only if there exist 
$ x,y\in K$ such that $[x,y] \neq 1$.

$(\Rightarrow)$ Let us assume that $[x,y] = 1$ for all $x,y \in K$. 
Therefore subgroup $\langle K\rangle$ is a finitely generated abelian group. 
Thus it is amenable and $K\notin \mathfrak{W}_{BT}$.

$(\Leftarrow)$ Let us assume that there exist $x,y \in K$ with $[x,y] \neq 1$. 
Since $x,y,[x,y]$ are nontrivial elements of $G$ we have  
$\phi: G\rightarrow F_2$ such that 
$\phi(x)\neq\phi(y)\neq\phi([x,y])\neq 1$. 
Clearly, $\langle\phi(x), \phi(y)\rangle$ is a free group of rank $2$. 
Thus $\langle x,y \rangle$ is also a free subgroup of rank $2$. 
It remains to apply Proposition \ref{fif} exactly as in the proof of Theorem \ref{fg}. 

\end{proof}

\newpage


\section*{Supplementary material}

**Theorem 8.5.** *The family $\mathfrak{W}_{BT}$ is computable for any computable fully residually free group.*

*Proof.* Let $(G, \nu)$ be a computable fully residually free group. Since $(G, \nu)$ is computable, it suffices to show that $K \in \mathfrak{W}_{BT}$ if and only if there exist $x, y \in K$ such that $[x, y] \neq 1$.

($\Rightarrow$) Let us assume that $[x, y] = 1$ for all $x, y \in K$. Therefore subgroup $\langle K \rangle$ is a finitely generated abelian group. Thus it is amenable and $K \notin \mathfrak{W}_{BT}$.

($\Leftarrow$) Let us assume that there exist $x, y \in K$ with $[x, y] \neq 1$. Since $x, y, [x, y]$ are nontrivial elements of $G$ we have $\phi : G \to F_2$ such that $\phi(x) \neq \phi(y) \neq \phi([x, y]) \neq 1$. Clearly, $\langle \phi(x), \phi(y) \rangle$ is a free group of rank 2. Thus $\langle x, y \rangle$ is also a free subgroup of rank 2. It remains to apply Proposition 8.4 exactly as in the proof of Theorem 8.3. □

## 9 Undecidable amenability

The following theorem is a joint observation with A. Ivanov. It is based on Corollary 10.5 of Appendix below.

**Theorem 9.1.** *There is a finitely presented computable group $H$ such that the problem if a finite subset of $H$ generates an amenable subgroup is not decidable. In particular the family $\mathfrak{W}_{BT}$ of $H$ is not computable.*

*Proof.* Let $G$ be a computable group given by Corollary 10.5. Our construction starts with the following statement of A. Darbinyan.

> For any computable group $G$ there is a two-generated group $G_1 = \langle c, s \rangle$ with decidable word problem and a computable embedding $\phi_1 : G \to G_1$.

This is a part of the statement of Theorem 1 of [10] and the definition of $\phi_1$ at line 19 on p.4927 of that paper. We only add that applying this theorem

- we consider our $G$ with respect to the generators given by any enumeration of $G = \{a^{(1)}, a^{(2)}, \ldots\}$,
- $\phi_1$ is denoted by $\phi$ in [10], and
- the definition of $\phi$ in [10] should be corrected by $a^{(i)} \to [c, c^{s^{2^{i-1}}}]$.

It is clear that the problem if a finite subset of $G_1$ generates an amenable subgroup is not decidable. Theorem 6 of [9] states that

> a finitely generated group $G_1$ with decidable word problem can be embedded into a finitely presented group with decidable word problem.

Applying this theorem we obtain the group $H$ as in in the formulation. Indeed the embedding $\phi_2 : G_1 \to H$ is obviously defined by the restriction of $\phi_2$ to the finite set of generators of $G_1$. Thus $\phi_2$ is computable. Having this we see that the problem if a finite subset of $H$ generates an amenable subgroup is not decidable. The rest follows from Proposition 8.4. □

It is worth noting that in the group $G$ given by Corollary 10.5 the condition to be amenable coincides with the condition to be without non-abelian free subgroup. Thus the theorem holds in the formulation where the property "amenable" is replaced by "to have a non-abelian free subgroup".

## 10 Appendix. A computable group with undecidable amenability. By Aleksander Iwanow (Ivanov)

In this appendix we do not assume that the reader has any special knowledge in recursion theory. We give all necessary definitions [1] making the material available for group theorists.

Let $F_2$ be a 2-generated free group with the basis $\{a, b\}$. Let $\nu_0 : \omega \to F_2$ be a 1-1-numbering of $F_2$ so that the graph of the multiplication is a decidable relation on the numbers of elements of $F_2$. The existence of such a numbering follows from decidability of the word problem in free groups.

---
[1] see also Section 2.1 above

We define a sequence $(G_i, a_i, b_i)$, $i \in \omega$, of 2-generated groups as follows. For even $i$ we assume that $(G_i, a_i, b_i) = (F_2, a, b)$ and $\nu_i = \nu_0$. For odd $i$ we assume that $G_i$ is finite and the Cayley graph of $(G_i, a_i, b_i)$ has the same $i$-balls of 1 with the Cayley graph of $(F_2, a, b)$. This exactly means that $(F_2, a, b)$ is the limit of the sequence of marked groups $(G_i, a_i, b_i)$, $i \in \mathbb{N} \setminus 2\mathbb{N}$, in Grigorchuk's topology. Existence of such a sequence follows from residual finiteness of $F_2$. The constructivization $\nu_i : \omega \to G_i$ is the $mod|G_i|$-map, where $G_i$ is identified with $\{0, \ldots, |G_i| - 1\}$ in a fixed way.

Let $G = \sum_{i \in \omega} G_i$. Let

- $Pr_i(x, y)$ be the relation from $G \times G$ that $y$ is the $G_i$-projection of $x$, $i \in \omega$;

- $R \subset G \times G$ be the binary relation consisting of all pairs $(g_1, g_2)$ such that the tuple $g_1, g_2$ is not a free basis of a 2-generated free group.

*Remark* 10.1. It is easily seen that a pair $(g_1, g_2)$ satisfies $R$ in $G$ if and only if the subgroup $\langle g_1, g_2 \rangle$ is amenable. When it is amenable its even projections must be cyclic.

We define a numbering $\nu : \omega \to G$ as follows. We fix a 1-1-numbering of all finite subsets of $\omega \times \omega$ with pairwise distinct first coordinates. Given $k$ let $\{(k_1, l_1), \ldots, (k_s, l_s)\}$ be the subset with the number $k$ where $k_1 < \ldots < k_s$. Let $\nu(k) = \nu_{k_1}(l_1) + \ldots + \nu_{k_s}(l_s)$.

**Lemma 10.2.** *The structure $(G, \cdot, 1, R, \{Pr_i\}_{i \in \omega}, \nu)$ is computable (in Russian terminology: constructive).*

*Proof.* The structure $(G, \cdot, 1, \{Pr_i\}_{i \in \omega}, \nu)$ is computable by the definition of $\nu$ and the observation that an equation is satisfied in $G$ if and only if it is satisfied in all $G_i$.

To see that the relation $R$ is computable in $G$ note that $(\nu(n_1), \nu(n_2)) \notin R$ if and only if there is an even index $t$ appearing as the first coordinate both in the subsets presented by numbers $n_1$ and $n_2$ so that the $t$-th coordinates of $\nu(n_1)$ and $\nu(n_2)$ do not commute. □

Having this lemma it is standard that the structure $(G, \cdot, 1, R, \{Pr_i\}_{i \in \omega})$ is computable with respect to an 1-1-numbering. Thus we may assume below that $\nu$ is 1-1.

**Theorem 10.3.** *There is an 1-1-numbering $\nu'$ of the group $G$ such that $(G, \nu')$ is a computable group where the relation $R$ is not computable with respect to $\nu'$.*

We apply the theory of **intrinsically computable relations**, [1], [2], [3]. In fact we use the advanced version of it from [2]. We remind the reader that the binary relation $\bar{a} \leq_0 \bar{b}$ on tuples of the same length $n$ from a computable structure $M$ is the property that $M \models \phi(\bar{a})$ implies $M \models \phi(\bar{b})$ where $\phi(x_1, \ldots, x_n)$ is atomic or the negation of an atomic formula with the Gödel number $< n$. For tuples with $|\bar{a}| = n \leq |\bar{b}|$ this relation means that $\bar{a} \leq_0 \bar{b}|_n$. The following statement is a modest part of Theorem 2.1 from [2].

> Let $(M, R)$ be a computable structure. Suppose that for any $\bar{c}$ there is a tuple $\bar{a} \notin R$ such that for any tuple $\bar{a}_1$ there exist $\bar{a}' \in R$ and $\bar{a}'_1$ such that $\bar{c}\bar{a}\bar{a}_1 \leq_0 \bar{c}\bar{a}'\bar{a}'_1$ in $M$. Then for every computably enumerable set $C$ there is an isomorphism $f$ onto a computable structure $M'$ such that $C$ and $f(R)$ are of the same Turing degree.

*Remark* 10.4. The formulation of Theorem 2.1 from [2] contains an assumption of friendliness of the structure. This is a notion which is defined by induction, where 1-**friendliness** of a structure means that the relation $\leq_0$ is computably enumerable on the set of tuples of arbitrary length. This condition is always satisfied in a computable structure. Since the version of the theorem which we use is restricted to 1-friendly structures, we can omit the assumption of friendliness.

*Proof of Theorem 10.3.* Let us verify the condition of Theorem 2.1 from [2]. Let $t_1 < t_2 \ldots < t_r$ be the indices of those $G_i$ where elements from $\bar{c}$ have non-trivial projections. These numbers can be found algorithmically by Lemma 10.2.

Let $t_{r+1}$ be the first even index greater than $t_r$. We define $\bar{a}$ to be $(a_{t_{r+1}}, b_{t_{r+1}})$, the free basis of $G_{t_{r+1}}$. We consider it as a pair of elements of $G$.

Let $\bar{a}_1$ be any tuple from $G$, and let $n$ be the length of $\bar{c}\bar{a}\bar{a}_1$. We want to find $\bar{a}' \in R$ and $\bar{a}'_1$ as in the formulation. In particular verifying $\bar{c}\bar{a}\bar{a}_1 \leq_0 \bar{c}\bar{a}'\bar{a}'_1$ we only consider formulas of Gödel numbers $< n$. We may suppose that these formulas are as follows:

$$\{w_i(\bar{z}, \bar{x}, \bar{x}_1) = 1 : i \in I_1\} \cup \{v_i(\bar{z}, \bar{x}, \bar{x}_1) \neq 1 : i \in I_2\},$$

14

where $w_i$ and $v_i$ are group words. For a word $w(\bar{z}, \bar{x}, \bar{x}_1)$ let $w(\bar{c}, \bar{a}, \bar{a}_1)(t)$ be the word written in the generators $a_t, b_t$ which is obtained by the substitution of the $G_t$-projections of elements from $\bar{c}\bar{a}\bar{a}_1$ into $w(\bar{z}, \bar{x}, \bar{x}_1)$ (before reductions). Let

$$n_0 = \max \bigcup \{\{|w_i(\bar{c}, \bar{a}, \bar{a}_1)(t)| : i \in I_1\} \cup \{|v_i(\bar{c}, \bar{a}, \bar{a}_1)(t)| : i \in I_2\} : t \in \{t_1, \ldots t_{r+1}\}\} + 1.$$

Let $\hat{t}$ be the first odd index greater than $\max(\mathsf{supp}(\bar{c}\bar{a}\bar{a}_1))$, such that the $n_0$-ball of 1 in the Cayley graph of $G_{t_{r+1}}$ is isometric to the $n_0$-ball of 1 in the Cayley graph of $G_{\hat{t}}$. Let us define $\bar{a}'\bar{a}'_1$ as follows:

$$\bar{a}'\bar{a}'_1(t) = \bar{a}\bar{a}_1(t) \text{ when } t \notin \{t_{r+1}, \hat{t}\}, \text{ and } \bar{a}'\bar{a}'_1(t_{r+1}) = \bar{1}, \text{ and the words of the sequence } \bar{a}'\bar{a}'_1(\hat{t})$$

coincide with ones of the sequence $\bar{a}\bar{a}_1(t_{r+1})$ under the correspondence $(a_{\hat{t}}, b_{\hat{t}}) \leftrightarrow (a_{t_{r+1}}, b_{t_{r+1}})$.

It is clear that $\bar{a}' \in R$. Since the sequences $\bar{c}\bar{a}\bar{a}_1$ and $\bar{c}\bar{a}'\bar{a}'_1$ coincide on the sets of indices

$$\mathsf{supp}(\bar{c}\bar{a}\bar{a}_1) \setminus \{t_{r+1}, \hat{t}\} = \mathsf{supp}(\bar{c}\bar{a}'\bar{a}'_1) \setminus \{t_{r+1}, \hat{t}\},$$

their realizations on the formulas of Gödel numbers $< n$ are equivalent on this part of the support. Note that $t_{r+1} \notin \mathsf{supp}(\bar{c}\bar{a}'\bar{a}'_1)$ and $\hat{t} \notin \mathsf{supp}(\bar{c}\bar{a}\bar{a}_1)$. Thus to obtain the result it suffices to note that for any word $w$ appearing in the formula of the Gödel number $< n$ the equality $w(\bar{c}, \bar{a}, \bar{a}_1)(t_{r+1}) = 1$ is equivalent to $w(\bar{c}, \bar{a}', \bar{a}'_1)(\hat{t}) = 1$. The latter follows from the choice of $n_0$ and $\hat{t}$. We now apply Theorem 2.1 from [2]. The copy $(G, \nu')$ of $(G, \nu)$ under the isomorphism $f$ from the formulation gives the result. □

Applying Remark 10.1 and Proposition 8.4 we obtain the following statement.

**Corollary 10.5.** *There is a computable group $G$ such that the problem if a finite subset of $G$ generates an amenable subgroup is not decidable. Moreover the set $\mathfrak{W}_{BT}$ is not computable in this group.*

## 10.1 Comments

A relation $R$ on a computable structure $M$ is called **intrinsically computable** if it is computable in any computable presentation of $M$ (i.e. with respect to any 1-1-constructivization of $M$). The following statement is Theorem 3.1 from [1] (originally proved in [3]). It is a slightly simplified version of Theorem 2.1 from [2], which we used above.

> Let $(M, R)$ be a computable structure whose existential diagram (i.e. the set of existential formulas with parameters which hold in $(M, R)$) is computable. Then $R$ is intrinsically recursive on $M$ if and only if both $R$ and its complement are formally computably enumerable on $M$.

In this formulation formal c.e. means that $R$ is equivalent to a disjunction (possibly infinite) $\bigvee \phi_n(\bar{x}, \bar{c})$ of existential formulas over a tuple $\bar{c}$. The following proposition is very close to the proof of Theorem 10.3.

**Proposition 10.6.** *Let $G$ and $R$ be the group and relation defined above. Then $\neg R$ is not formally computably enumerable on $G$.*

*Proof.* Let $\bar{\mathbf{c}} \in G$ and let $\bigvee \phi_n(x_1, x_2, \bar{\mathbf{c}})$ be a disjunction of existential formulas of the group theory language. If $\neg R$ is defined by this disjunction then each $\phi_n(x_1, x_2, \bar{\mathbf{c}})$ implies that $x_1$ and $x_2$ is a free basis. Moreover if $t$ is an even index outside the support of $\bar{\mathbf{c}}$ and $a, b$ is a free basis of $G_t$ then the tuple $\mathbf{a}, \mathbf{b}, \bar{1}$ with

$$\mathbf{a}(t) = a \ , \ \mathbf{a}(i) = 1 \text{ for } i \neq t,$$
$$\mathbf{b}(t) = b \ , \ \mathbf{b}(i) = 1 \text{ for } i \neq t,$$

realizes some $\phi_n(x_1, x_2, \bar{y})$. Fixing this $n$ assume that $\phi_n(x_1, x_2, \bar{y}) = \exists \bar{z} \phi'(x_1, x_2, \bar{y}, \bar{z})$, where $\phi'$ is quantifier free. Let $m$ be a number which is greater than the sum of the lengths of the words appearing in $\phi'(x_1, x_2, \bar{y}, \bar{z})$. Take $\bar{\mathbf{d}}$ realizing $\phi'(\mathbf{a}, \mathbf{b}, \bar{1}, \bar{z})$ in $G$. Let $\bar{d}$ be the projection of $\bar{\mathbf{d}}$ to $G_t$. By the choice of $G_i$ with odd $i$ there is a sufficiently large odd index $l$ outside the support of $\bar{\mathbf{d}}$ and a tuple $a', b', \bar{1}, \bar{d}' \in G_l$ such that the $m$-ball in the Cayley graph of $G_t$ with respect to the generators $a, b, \bar{d}$ is isomorphic to the corresponding ball of $(G_l, a', b', \bar{d}')$ under the map $a, b, \bar{1}, \bar{d} \to a', b', \bar{1}, \bar{d}'$. Let us define

$$\mathbf{a}'(l) = a' \ , \ \mathbf{a}'(i) = 1 \text{ for } i \neq l,$$

15

$$\mathbf{b}'(l) = b' \,, \mathbf{b}'(i) = 1 \text{ for } i \neq l,$$

and $\bar{\mathbf{d}}'(l) = \bar{d}'$, $\bar{\mathbf{d}}'(t) = \bar{1}$, $\bar{\mathbf{d}}'(i) = \bar{\mathbf{d}}(i)$ for $i \notin \{l, t\}$. Then the tuple $\mathbf{a}', \mathbf{b}', \bar{1}, \bar{\mathbf{d}}'$ satisfies $\phi'(x_1, x_2, \bar{y}, \bar{z})$ and the tuple $\mathbf{a}', \mathbf{b}', \bar{1}$ satisfies $\phi_n(x_1, x_2, \bar{y})$ in $G$. Since $\mathbf{a}'$ and $\mathbf{b}'$ do not generate a free group, we obtain a contradiction. □

This proposition suggests that our group $G$ and the relation $R$ also satisfy the conditions of Theorem 3.1 from [1] (they are stronger than ones used in the proof of Theorem 10.3). However the author was not able to prove that the existential diagram of the structure $(G, \cdot, 1, R)$ is computable (this is the only remaining task).

The attempts which were made lead us to the following problem.

Is there a family of finite two-generated groups

$$\mathcal{G} = \{G_l = \langle a_l, b_l \rangle : l \in \omega\}$$

such that the universal (or elementary) theory of $\mathcal{G}$ is decidable and $(F_2, a, b)$ is a limit group of this family in the Grigorchuk topology?

It is worth noting that the elementary theory of $(F_2, a, b)$ is decidable, [13].



\section*{Acknowledgements} 
The authors are grateful to M. Cavaleri, T. Ceccherini-Silberstein and L. Ko{\l}odziejczyk for reading the paper and helpful remarks.  
We also thank M. Sapir for right advice concerning Higman's embedding.  
The authors are grateful to the referee for remarks which substantionally improved the exposition. 

This research was partially supported by (Polish) Narodowe Centrum Nauki, grant UMO--2018/30/M/ST1/00668.

\end{document}